\pdfoutput=1
\documentclass{amsart}

\usepackage[utf8]{inputenc}   
\usepackage[OT2, T1]{fontenc}
\usepackage[english]{babel}

\usepackage{amsmath}              
\usepackage{amscd}
\usepackage{amssymb}              

\usepackage{float}
\usepackage{enumerate}
\usepackage{placeins}

\usepackage{graphicx}

\newcommand{\Z}{\mathbb{Z}}                    
\newcommand{\Sp}{\operatorname{Sp}}

\newcommand{\GL}{\operatorname{GL}}
\newcommand{\SO}{\operatorname{SO}}

\theoremstyle{plain}
\newtheorem{lause}{Theorem}[section]

\newtheorem{lemma}[lause]{Lemma}

\newtheorem{prob}[lause]{Problem}

\theoremstyle{definition}
\newtheorem{remark}[lause]{Remark}
\newtheorem{conjecture}[lause]{Conjecture}

\usepackage{url}

\usepackage{hyphenat}
\usepackage{chngcntr}
\counterwithin{table}{section}

\begin{document}

\title[A counterexample to a conjugacy conjecture of Steinberg]{A counterexample to a conjugacy conjecture of Steinberg}

\author{Mikko Korhonen}
\address{School of Mathematics, The University of Manchester, Manchester M13 9PL, United Kingdom}
\email{korhonen\_mikko@hotmail.com}

\begin{abstract}
Let $G$ be a semisimple algebraic group over an algebraically closed field of characteristic $p \geq 0$. At the 1966 International Congress of Mathematicians in Moscow, Robert Steinberg conjectured that two elements $a, a' \in G$ are conjugate in $G$ if and only if $f(a)$ and $f(a')$ are conjugate in $\GL(V)$ for every rational irreducible representation $f: G \rightarrow \GL(V)$. Steinberg showed that the conjecture holds if $a$ and $a'$ are semisimple, and also proved the conjecture when $p = 0$. In this paper, we give a counterexample to Steinberg's conjecture. Specifically, we show that when $p = 2$ and $G$ is simple of type $C_5$, there exist two non-conjugate unipotent elements $u, u' \in G$ such that $f(u)$ and $f(u')$ are conjugate in $\GL(V)$ for every rational irreducible representation $f: G \rightarrow \GL(V)$.
\end{abstract}

\maketitle

\section{Introduction}

Let $G$ be a semisimple algebraic group over an algebraically closed field $K$ of characteristic $p \geq 0$. At the 1966 International Congress of Mathematicians in Moscow, Steinberg proposed the following conjecture \cite[Problem (4)]{SteinbergICM}.

\begin{conjecture}[Steinberg]\label{conjecture:steinberg}
Two elements $a$ and $a'$ of $G$ are conjugate in $G$ if and only if $f(a)$ and $f(a')$ are conjugate in $\GL(V)$ for every rational irreducible representation $f: G \rightarrow \GL(V)$.
\end{conjecture}

One motivation for the conjecture, observed by Steinberg in \cite{SteinbergICM}, is that a positive answer to the conjecture would imply that $G$ has only a finite number of unipotent conjugacy classes. When Steinberg posed Conjecture \ref{conjecture:steinberg}, the finiteness of the number of unipotent conjugacy classes was known in good characteristic \cite[Proposition 5.2]{RichardsonFinite67}, and it was eventually shown to be true in general by Lusztig \cite{LusztigFinite}. Lusztig's proof, based on the theory of Deligne-Lusztig characters of finite groups of Lie type, still remains the only known uniform proof of the result in characteristic $p > 0$.

The main purpose of this paper is to provide a counterexample to Conjecture \ref{conjecture:steinberg}. Specifically, our main result is the following, which is given by Theorem \ref{thm:mainresult}:

\begin{lause}
Let $p = 2$ and let $G$ be simple of type $C_5$. Then there exist two non-conjugate unipotent elements $u, u' \in G$ such that $f(u)$ and $f(u')$ are conjugate in $\GL(V)$ for every rational irreducible representation $f: G \rightarrow \GL(V)$.
\end{lause}

We will also show that in the case of unipotent elements, our counterexample is minimal in the sense that up to isogenies, there are no other examples for simple groups of rank at most $5$ (Theorem \ref{thm:rankatmost5}).

It remains an open question to determine when exactly the conclusion of Conjecture \ref{conjecture:steinberg} holds. Steinberg proved Conjecture \ref{conjecture:steinberg} in the case where $a$ and $a'$ are both semisimple \cite[6.6]{SteinbergRegular}, and also in the case where $p = 0$ \cite[Theorem 3]{SteinbergConjecture}. Currently it is not known whether the conjecture holds in the case where $p$ is a \emph{good prime} for $G$, and the following related conjecture also remains open.

\begin{conjecture}\label{conjecture:conjecture_more_general}
Two elements $a$ and $a'$ of $G$ are conjugate in $G$ if and only if $f(a)$ and $f(a')$ are conjugate in $\GL(V)$ for every rational representation $f: G \rightarrow \GL(V)$.
\end{conjecture}

Obviously a counterexample to Conjecture \ref{conjecture:conjecture_more_general} would also give a counterexample to Conjecture \ref{conjecture:steinberg}. However,  since rational representations of $G$ are usually not completely reducible, it does not seem clear that the converse is true. Indeed, we do not know whether our counterexample to Conjecture \ref{conjecture:steinberg} gives a counterexample to Conjecture \ref{conjecture:conjecture_more_general} --- see Remark \ref{remark:moregeneralconjecture}.

The work done in this paper is strongly based on calculations implemented in MAGMA \cite{MAGMA}, and we discuss the computational aspects in Section \ref{section:computations}. Some computations similar to ours have also been done in previous work of Lawther \cite{Lawther}, who considered exceptional $G$ and the action of unipotent elements of $G$ on the adjoint and minimal modules of $G$. 


\section{Notation}

We fix the following notation and terminology. Throughout the text, let $K$ be an algebraically closed field of characteristic $p > 0$. All the groups that we consider are linear algebraic groups over $K$, and $G$ will always denote a simple linear algebraic group over $K$. We say that $p$ is \emph{good for $G$}, if $G$ is simple of type $A_l$; if $G$ is simple of type $B_l$, $C_l$, or $D_l$, and $p > 2$; if $G$ is simple of type $G_2$, $F_4$, $E_6$, or $E_7$, and $p > 3$; or if $G$ is simple of type $E_8$ and $p > 5$. Otherwise we say that the prime $p$ is \emph{bad for $G$}.

For the unipotent conjugacy classes in $G$, we will use the labeling given by the Bala-Carter classification of unipotent conjugacy classes \cite{BalaCarter1, BalaCarter2}, which is valid in good characteristic by Pommerening's theorem \cite{Pommerening1, Pommerening2}. When $G$ is simple of exceptional type and $p$ is bad for $G$, the labeling is extended to bad characteristic as well, although there are more unipotent classes and we label them as in \cite{Lawther}. See for example \cite[pp. 4128--4129]{Lawther} for further explanation.

We fix a maximal torus $T$ of $G$ with character group $X(T)$. Fix a base $\Delta = \{ \alpha_1, \ldots, \alpha_l \}$ for the root system of $G$, where $l$ is the rank of $G$. Here we use the standard Bourbaki labeling of the simple roots $\alpha_i$, as given in \cite[11.4, p. 58]{Humphreys}. We denote the dominant weights with respect to $\Delta$ by $X(T)^+$, and the fundamental dominant weight corresponding to $\alpha_i$ is denoted by $\varpi_i$. For a dominant weight $\lambda \in X(T)^+$, we can write $\lambda = \sum_{i = 1}^l a_i \varpi_i$ for unique integers $a_i \in \Z_{\geq 0}$. We say that $\lambda$ is \emph{$p$-restricted}, if $0 \leq a_i \leq p-1$ for all $1 \leq i \leq l$. The longest element in the Weyl group of $G$ is denoted by $w_0$.  

We use the notation $L_G(\lambda)$ for the irreducible $G$-module with highest weight $\lambda \in X(T)^+$, and denote by $V_G(\lambda)$ the Weyl module of highest weight $\lambda \in X(T)^+$. 



Throughout, $V$ will denote a finite-dimensional vector space over $K$. Let $u \in \GL(V)$ be a unipotent linear map. It will often be convenient for us to describe the action of $u$ on a representation in terms of $K[u]$-modules. Suppose that $u$ has order $q = p^t$. Then there exist exactly $q$ indecomposable $K[u]$-modules which we will denote by $J_1$, $J_2$, $\ldots$, $J_q$. Here $\dim J_i = i$ and $u$ acts on $J_i$ as a full Jordan block. We use the notation $r \cdot J_n$ for the direct sum $J_n \oplus \cdots \oplus J_n$, where $J_n$ occurs $r$ times. For the Jordan form of $u$ acting on $V$ we will use the notation $(d_1^{n_1}, \ldots, d_t^{n_t})$ when $V \downarrow K[u] = n_1 \cdot J_{d_1} \oplus \cdots \oplus n_t \cdot J_{d_t}$ for integers $1 \leq d_1 < \cdots < d_t$ and $n_i > 0$.

Let $V$ be a $G$-module with associated representation $f: G \rightarrow \GL(V)$. We say that two elements $g, g' \in G$ \emph{act similarly on $V$} if $f(g)$ and $f(g')$ are conjugate in $\GL(V)$. If $f(g)$ and $f(g')$ are not conjugate in $\GL(V)$, we say that $f$ (or $V$) \emph{separates} the conjugacy classes of $g$ and $g'$.

\section{On computations}\label{section:computations}

When studying Conjecture \ref{conjecture:steinberg}, one has to consider the following problem.

\begin{prob}\label{prob:difficult}
Let $\lambda \in X(T)^+$ be a dominant weight and $g \in G$. What is the Jordan normal form of $g$ acting on $L_G(\lambda)$?
\end{prob}

Currently there is no general answer to Problem \ref{prob:difficult}. One basic obstacle here is that even determining the dimensions of the irreducible representations of $G$ is a difficult open problem that is expected to remain open for some time, so in general we cannot even expect an answer in the case where $g = 1$. However, given a specific $G$, element $g \in G$, and highest weight $\lambda$, one can use computational methods to construct the matrix of $g$ acting on $L_G(\lambda)$ with respect to some basis; standard algorithms then allow one to find the Jordan normal form of this matrix. We have used MAGMA for the computer calculations done in this paper, the purpose of this section is to outline the methods used.

Throughout this section we will assume that $\lambda$ is $p$-restricted, since with Steinberg's tensor product theorem \cite[Theorem 6.1]{SteinbergREP} any computation can be reduced to this case. Furthermore, we will assume that $g$ is unipotent since all of the computations in this paper are given in this case; similar computations are possible in the general case as well. Note that in the unipotent case the question is what are the Jordan block sizes and their multiplicities for $g$ acting on $L_G(\lambda)$.

\subsection{Chevalley construction}\label{subsec:chevalleycomp}

We first describe a general method, where the idea is to follow the construction of Chevalley groups from complex semisimple Lie algebras, as in \cite{SteinbergNotesAMS}. To begin with, we have algorithms to construct the irreducible representation of highest weight $\lambda$ for the corresponding Lie algebra in characteristic zero \cite{deGraaf_LIE}. Then by exponentiation and reduction modulo $p$, one can construct $V_G(\lambda)$ \cite[p. 1488]{CMT_computing}. This construction of $V_G(\lambda)$ has been implemented in MAGMA, and it allows one to find the action of each root element $x_{\alpha}(c)$ with respect to a fixed basis of weight vectors in $V_G(\lambda)$. Given $V_G(\lambda)$ in this manner, one can implement methods to construct $L_G(\lambda)$ as the unique simple quotient of $V_G(\lambda)$, or equivalently as the unique simple submodule of $V_G(-w_0 \lambda)^*$. For more details on this computation of $L_G(\lambda)$, see for example \cite[Section 5.5]{deGraaf_BOOK}. 

Finally, one can find a representative for the conjugacy class of $g$ in $G$ as a product of root elements of the form $x_\alpha(\pm 1)$ and thus compute the matrix of $g$ acting on $L_G(\lambda)$. For computations when $G$ is simple of exceptional type, we have used representatives from the tables in \cite{Simion}, which are based on \cite{LiebeckSeitzClass} and some computations due to Ross Lawther. For $G$ simple of classical type, representatives in terms of root elements of the form $x_{\alpha}(\pm 1)$ are well known, but do not seem to appear explicitly in the literature. For $p \neq 2$, see \cite[Lemma 2.24, Proof of Lemma 2.39]{SuprunenkoMin} for one method of finding representatives. With similar ideas, explicit representatives can be constructed in characteristic $p = 2$ as well, but we shall omit the details from this paper.

\subsection{Computing with finite groups of Lie type}\label{section:lietypecomp}

The computation based on the Chevalley construction works in general, but it requires the computation of a representation in characteristic zero, which in our setting can be slow and comes with a high memory cost for large-dimensional representations. For our purposes, it is sometimes better to start in positive characteristic $p > 0$ in the first place. Many of our computations have been done with the corresponding finite group of Lie type $G(p) := G(\mathbb{F}_p)$ in small characteristic (usually $p = 2$), and in this case matrix computations and the ``Meataxe'' algorithm are especially efficient \cite{MEATAXE_Parker} \cite{MEATAXE_HoltRees} \cite{MEATAXE_Holt} \cite{MEATAXE_IvanyosLux}. Note that since $g$ is unipotent, by replacing $g$ with a conjugate, we may assume that $g$ is contained in the finite group $G(p)$. This is well known, and follows for example from the fact mentioned above that the conjugacy class of $g$ in $G$ has a representative which is a product of root elements of the form $x_{\alpha}(\pm 1)$.

Also useful for us is the result of Steinberg \cite{Steinberg_GEN} that $G(p)$ can always be generated by two elements. This speeds up the computations and reduces the amount of memory needed compared to the Chevalley construction, where computations involving many root elements are needed. Explicit matrices corresponding to the generators provided by Steinberg are given in \cite{GEN_classical1}, \cite{GEN_classical2}, \cite{GEN_exceptional}, and have been implemented in MAGMA. This gives us $G(p) = \langle A, B \rangle$ for explicit matrices $A$ and $B$ acting on an $\mathbb{F}_p[G(p)]$-module $W$. For classical types $W$ is the natural module, while for exceptional types $W$ corresponds to the non-trivial Weyl module of minimal dimension.

We now describe how the action of $g$ on $L_G(\lambda)$ is computed. First, by another result due to Steinberg, every irreducible $KG$-module with $p$-restricted highest weight $\lambda$ restricts irreducibly to $G(p)$, and furthermore every irreducible $\mathbb{F}_p[G(p)]$-module is absolutely irreducible \cite[Corollary 7.5]{SteinbergREP}. Thus as $K[G(p)]$-modules we have $L_G(\lambda) \cong V \otimes_{\mathbb{F}_p} K$ for an irreducible $\mathbb{F}_p[G(p)]$-module $V$, so it will suffice to compute the matrix of $g$ acting on $V$ with respect to some basis.

Before computing $V$, as a preliminary step one should find a representative for the conjugacy class of $g$ in $G$ as a relatively short word in the generators $A$ and $B$. Then once we have computed matrices representing the action of $A$ and $B$ on $V$, it will be efficient to find the Jordan normal form of $g$ acting on $V$. For our computations, a simple random search among words in $A$ and $B$ has been sufficient. For each randomly chosen word $$w = A^{i_1} B^{j_1} A^{i_2} B^{j_2} \cdots$$ we test for conjugacy with $g$ by first checking if $w$ is unipotent with same order as $g$. Then if $w$ is unipotent, we can usually recognize the unipotent conjugacy class of $w$ in $G$ from its Jordan normal form on $W$, although in small characteristic some additional information is needed. For classical groups in characteristic $p = 2$, we need to use the corresponding alternating or quadratic form on the natural module, see \cite{Hesselink}. For $G$ simple of exceptional type, we have used \cite[Section 3]{LawtherFusion} to distinguish between the unipotent conjugacy classes. As an example, we have used this method in the case where $G$ is of type $E_8$ and $p = 2$ to find representatives for all unipotent conjugacy classes of $G$ as words in the generators $A$ and $B$ of $E_8(2)$.

For the construction of $V$, in the cases that we consider, we are able to construct $V$ as the unique composition factor of its dimension within certain modules involving tensor products and exterior powers of $W$. There are methods in the MAGMA library for computation of exterior powers and tensor products of $\mathbb{F}_p[G(p)]$-modules, and also an implementation of the Meataxe algorithm which allows one to compute the composition factors of a given $\mathbb{F}_p[G(p)]$-module. Finding $V$ as the unique composition factor of its dimension is justified by calculations with characters of $KG$-modules, and the dimensions and weight multiplicities of relevant $\mathbb{F}_p[G(p)]$-modules found in \cite{LubeckWebsite}. 

\FloatBarrier
\section{Simple groups of exceptional type}

In this section, we consider Conjecture \ref{conjecture:steinberg} in the case where $G$ is simple of exceptional type and $a, a' \in G$ are unipotent elements. In this setting, in good characteristic Conjecture \ref{conjecture:steinberg} is seen to hold from computations of Lawther in \cite{Lawther}. With further computations, using the methods described in Section \ref{section:computations}, we can extend this to bad characteristic. We find that the only potential counterexamples to Conjecture \ref{conjecture:steinberg} in the unipotent case occur when $p = 2$ and $G$ is simple of type $E_7$ or $E_8$.

\begin{lause}\label{thm:exceptional_computations}
Let $G$ be simple of exceptional type and let $u, u' \in G$ be unipotent elements. Suppose that $f(u)$ and $f(u')$ are conjugate in $\GL(V)$ for every rational irreducible representation $f: G \rightarrow \GL(V)$. Then either $u$ and $u'$ are conjugate in $G$, or one of the following holds:

\begin{enumerate}[\normalfont (i)]
\item $G$ is simple of type $E_7$, $p = 2$, and the unipotent conjugacy classes of $u$ and $u'$ are $A_5A_1$ and $D_6(a_2)$.
\item $G$ is simple of type $E_8$, $p = 2$, and the unipotent conjugacy classes of $u$ and $u'$ are one of the following:
\begin{itemize}
\item $D_5(a_1)A_1$ and $(D_4A_2)_2$,
\item $A_5A_1$ and $D_6(a_2)$,
\item $E_7(a_3)$ and $D_7(a_1)_2$,
\item $E_6A_1$ and $E_7(a_2)$,
\item $E_7(a_1)$ and $E_8(b_4)$.
\end{itemize}
\end{enumerate}
\end{lause} 

\begin{proof}

To prove the theorem, without loss of generality we can assume that $G$ is simple of adjoint type. Consider the irreducible representation $L_G(\beta)$ with highest weight the highest short root $\beta$. Then except for the pairs of classes listed in Table \ref{table:exceptional_separation}, the representation $L_G(\beta)$ separates the unipotent conjugacy classes of $G$. Indeed, when $V_G(\beta) \cong L_G(\beta)$, this follows from the computations in \cite{Lawther}. If $V_G(\beta)$ is not irreducible, then it is well known that $(G,p)$ is $(G_2, 2)$, $(F_4, 3)$, $(E_6, 3)$, or $(E_7, 2)$; these cases can be handled by a computation with MAGMA (see Section \ref{subsec:chevalleycomp}).
 
What remains then is to consider the pairs of classes listed in Table \ref{table:exceptional_separation}. Except for the pairs listed in (i) and (ii), we find that the following irreducible representations separate the pairs of classes in Table \ref{table:exceptional_separation}: for type $G_2$, the representation $L_G(\varpi_2)$; for type $F_4$, the representation $L_G(\varpi_1)$; for type $E_6$, the representation $L_G(\varpi_1 + \varpi_6)$; for type $E_7$, the representation $L_G(\varpi_6)$; for type $E_8$, the representation $L_G(\varpi_1)$. This is seen by a computation with MAGMA, see Table \ref{table:combined}. In the case where $G$ is simple of type $F_4$ and $p = 3$, one could also refer to \cite[Table 4]{Lawther}.\end{proof}

\begin{remark}
We do not know whether the pairs listed in Theorem \ref{thm:exceptional_computations} give counterexamples to Conjecture \ref{conjecture:steinberg}. However, given the counterexample described in Section \ref{section:counterexample}, it would not be surprising if at least some of these pairs provide further counterexamples. This is further illustrated by our computer calculations, which we have used to verify that for $G = E_7$ and $p = 2$, unipotent elements in classes $A_5A_1$ and $D_6(a_2)$ act similarly on $L_G(\varpi_i)$ for $i \neq 4$. We have also verified that for $G = E_8$ and $p = 2$, the pairs of unipotent classes in Theorem \ref{thm:exceptional_computations} act similarly on $L_G(\varpi_7)$.
\end{remark}

\begin{table}[!htbp]
\centering
\renewcommand{\arraystretch}{1.2}
\caption{For adjoint $G$ simple of exceptional type, all pairs of unipotent conjugacy classes which are not separated by $L_G(\beta)$, where $\beta$ is the highest short root.}\label{table:exceptional_separation}
\begin{tabular}{| c | c | c l |}
\hline
$G$              & $\beta$ &  \multicolumn{2}{ l |}{Pairs of classes not separated by $L_G(\beta)$} \\
\hline
& & & \\[-1em]
$G_2$            & $\varpi_1$ & $p = 3$ & $(\tilde{A}_1)_3$, $\tilde{A}_1$ \\
& & &  \\
$F_4$            & $\varpi_4$ & $p = 3$ & $C_3$, $F_4(a_2)$ \\ 
                 &            &         & $\tilde{A_2}$, $\tilde{A_2}A_1$ \\
								 & & & \\
                 &            & $p = 2$ & $(\tilde{A_1})_2$, $\tilde{A_1}$ \\
                 &            &         & $(\tilde{A_2} A_1)_2$, $\tilde{A_2} A_1$ \\
								 &            &         & $(B_2)_2$, $B_2$ \\
                 &            &         & $(C_3(a_1))_2$, $C_3(a_1)$ \\			
																									
& & & \\
$E_6$            & $\varpi_2$ & $p = 2$ & $A_3A_1$, $A_2^2A_1$ \\
& & & \\
$E_7$            & $\varpi_1$ & $p = 3$ & $A_2A_1^3$, $A_2^2$ \\
                 & &                    & $D_5A_1$, $D_6(a_1)$ \\
								 & & & \\
								 &            & $p = 2$ & $D_6(a_2)$, $A_5A_1$ \\
								 &            &         & $A_3A_2$, $(A_3A_2)_2$ \\
& & & \\
$E_8$            & $\varpi_8$ & $p = 2$ & $A_5A_1$, $D_6(a_2)$ \\
                 &            &         & $D_5(a_1)A_1$, $(D_4A_2)_2$ \\				
                 &            &         & $D_4(a_1)A_1$, $(A_3A_2)_2$ \\	
							   &            &         & $D_6(a_1)$, $E_7(a_4)$ \\	
								 &            &         & $E_7(a_3)$, $(D_7(a_1))_2$ \\	
								 &            &         & $E_6A_1$, $E_7(a_2)$ \\	
								 &            &         & $E_7(a_1)$, $E_8(b_4)$ \\[5pt]	
\hline
\end{tabular}

\end{table}

\begin{table}[!htbp]
\centering
\renewcommand{\arraystretch}{1.2}
\caption{For adjoint $G$ simple of exceptional type, an irreducible $G$-module $L_G(\lambda)$ separating most pairs of unipotent classes listed in Table \ref{table:exceptional_separation}.}\label{table:combined}
\begin{tabular}{| c | c | c | l |}
\hline
$G$            & $\lambda$  & Class             & Jordan blocks on $L_G(\lambda)$ \\ \hline
& & & \\[-1em]
$G_2$, $p = 3$ & $\varpi_2$ & $(\tilde{A}_1)_3$ & $2^{2}$, $3$ \\
               &            & $\tilde{A}_1$ & $1^{3}$, $2^{2}$ \\
&&&\\											
							
$F_4$, $p = 3$ & $\varpi_1$ & $C_3$ & $1^{3}$, $3$, $6^{2}$, $8^{2}$, $9^{2}$ \\ 
               &            & $F_4(a_2)$ & $3^{2}$, $6^{2}$, $7$, $9^{3}$ \\ 
							&&&\\
               &            & $\tilde{A_2}$ & $1^{7}, 3^{15}$ \\
							 &            & $\tilde{A_2}A_1$ & $2^{2}, 3^{16}$ \\
&&&\\													
$F_4$, $p = 2$ & $\varpi_1$ & $(\tilde{A_1})_2$ & $1^{6}$, $ 2^{10}$ \\
               &            & $\tilde{A_1}$ & $1^{14}$, $ 2^{6}$ \\			
																										&&&\\
               &         & $(\tilde{A_2} A_1)_2$ & $2^{2}$, $ 3^{2}$, $ 4^{4}$ \\
               &         & $\tilde{A_2} A_1$ & $2^{6}$, $ 3^{2}$, $ 4^{2}$ \\	
													&&&\\
							 &         & $(B_2)_2$ & $2^{3}$, $4^{5}$ \\
						   &         & $B_2$ & $1^{4}$, $ 2$, $4^{5}$ \\		
							&&&\\
               &         & $(C_3(a_1))_2$ & $2^{3}$, $4^{5}$ \\		
							 &         & $C_3(a_1)$ & $1^{4}$, $2$, $4^{5}$ \\	
&&&\\							
$E_6$, $p = 2$ & $\varpi_1 + \varpi_6$ & $A_3A_1$ & $2^{10}$, $ 4^{138}$ \\			
               &         & 	$A_2^2A_1$ & $1^{4}$, $ 2^{6}$, $ 3^{4}$, $ 4^{136}$ \\			
&&&\\							
$E_7$, $p = 2$ & $\varpi_6$ & $A_3A_2$ & $2^{16}$, $3^{6}$, $4^{306}$ \\			
               &         & 	$(A_3A_2)_2$ & $2^{18}$, $3^{2}$, $4^{308}$ \\	
&&&\\
										
$E_7$, $p = 3$ & $\varpi_6$ & $A_2A_1^3$ & $3^{513}$ \\
							 &         &  $A_2^2$ & $1^{27}$, $3^{504}$ \\
							&&&\\
							 &         &  $D_5A_1$ & $1^{6}$, $ 3^{6}$, $ 6^{6}$, $ 8^{6}$, $ 9^{159}$ \\
							 &         &  $D_6(a_1)$ & $3^{8}$, $ 6^{8}$, $ 9^{163}$ \\
&&&\\

$E_8$, $p = 2$ & $\varpi_1$ & $D_4(a_1)A_1$ & $1^{12}, 2^{26}, 3^{14}, 4^{880}$ \\	
               &         & $(A_3A_2)_2$ & $1^{4}, 2^{30}, 3^{14}, 4^{880}$ \\	
							&&&\\
							 &         & $D_6(a_1)$ &  $1^{4}, 4^{15}, 6^{7}, 8^{440}$ \\	
						   &         &  $E_7(a_4)$ & $2^{2}, 4^{15}, 6^{7}, 8^{440}$ \\[5pt]	
							
							%
										%
															%
							%
							
\hline
\end{tabular}

\end{table}		
\FloatBarrier

\clearpage
\section{Counterexample}\label{section:counterexample}

\begin{table}[!htbp]
\centering
\renewcommand{\arraystretch}{1.2}
\caption{$G = C_5$, $p = 2$: Jordan block sizes of unipotent elements in classes $(2_1^2, 6_3)$ and $(2_0^2, 6_3)$ acting on $L_G(\lambda)$ for certain $\lambda \in X(T)^+$. Here we write $\lambda = a_1a_2a_3a_4a_5$ for $\lambda = a_1 \varpi_1 + a_2 \varpi_2 + a_3 \varpi_3 + a_4 \varpi_4 + a_5 \varpi_5$. The cases where $L_G(\lambda) \cong V_G(\lambda)$ have been marked with a (*).}\label{table:action_counterexample}
\begin{tabular}{| c | l | l |}
\hline
$\lambda$  & Jordan blocks on $L_G(\lambda)$ & $\dim L_G(\lambda)$ \\ \hline
$10000$    & $2^2, 6$ & $10^*$ \\
$01000$    & $1^2, 2^2, 6^5, 8$ & $44^*$ \\
$00100$    & $2^2, 6^8, 8^6$ & $100$ \\   
$00010$    & $2^8, 6^6, 8^{14}$ & $164$ \\
$00001$    & $2^4, 6^4$ & $32$ \\
$11000$    & $2^{16}, 6^{16}, 8^{24}$ & $320^*$ \\
$10100$    & $2^8, 4, 6^{23}, 8^{64}$ & $670$ \\
$10010$    & $2^{32}, 6^{32}, 8^{144}$ & $1408^*$ \\
$01100$    & $2^{36}, 6^{34}, 8^{304}$ & $2708$ \\
$01010$    & $2^6, 4^6, 6^{16}, 8^{374}$ & $3124$ \\
$00110$    & $2^{32}, 6^{32}, 8^{1072}$  & $8832$ \\
$11100$    & $2^{128}, 6^{128}, 8^{2112}$ & $17920^*$ \\
$11010$    & $2^{24}, 4^{36}, 6^{44}, 8^{2744}$ & $22408$ \\
$10110$    & $2^{38}, 4^{14}, 6^{54}, 7^{2}, 8^{6530}$ & $52710$ \\
$01110$    & $2^{64}, 6^{64}, 8^{22816}$ & $183040^*$ \\
$11110$    & $8^{131072}$ & $1048576^*$ \\
\hline
\end{tabular}
\end{table}

In this section, we give a counterexample to Conjecture \ref{conjecture:steinberg}. We begin with the following two lemmas.

\begin{lemma}\label{lemma:steinbergmodule_jordanblocks}
Let $\lambda = (p-1)\varpi_1 + \cdots + (p-1)\varpi_l$. Then for any unipotent element $u \in G$, we have $L_G(\lambda) \downarrow K[u] = N \cdot J_{p^k}$, where $p^k$ is the order of $u$ and $N = \dim L_G(\lambda) / p^k$.
\end{lemma}
      
\begin{proof}Without loss of generality, we can assume that $u$ is contained in the finite subgroup $G(p) := G(\mathbb{F}_p)$. Note that $L_G(\lambda)$ is the Steinberg module for $G(p)$, and by a result due to Steinberg its restriction to $G(p)$ is projective \cite[Theorem 8.2]{SteinbergREP} \cite[Theorem 1]{BrauerNesbitt}, see also \cite[§3]{Humphreys_SteinbergREP}. Then the restriction to the subgroup $H = \langle u \rangle$ is also projective \cite[Theorem 6, p. 33]{Alperin}. Since $J_{p^k}$ is the only projective indecomposable $KH$-module, the claim follows.\end{proof}

\begin{lemma}\label{lemma:char2typeDsteinberg}
Assume that $p = 2$ and let $G$ be simple of type $C_l$, $l \geq 2$. Let $\lambda = \varpi_1 + \cdots + \varpi_{l-1}$. If $u \in G$ is a unipotent element of order $2^k > 2$, then $L_G(\lambda) \downarrow K[u] = N \cdot J_{2^k}$, where $N = \dim L_G(\lambda) / 2^k$.
\end{lemma}

\begin{proof}
Without loss of generality we can assume that $G = \Sp(V, \beta)$, where $V$ is a vector space of dimension $2l$ and $\beta$ is a non-degenerate alternating bilinear form on $V$. It follows from \cite[Theorem 5]{Dye} that $u$ is contained in $H = \operatorname{O}(V, Q)$ for some non-degenerate quadratic form $Q$ which polarizes to $\beta$. Then we have $u^2 \in H^\circ = \SO(V, Q)$ since $H^\circ$ has index $2$ in $H$. 

We can also see $H^\circ$ as a subsystem subgroup of type $D_l$ in $G$, since $H^\circ$ is generated by short root subgroups of $G$. From this point of view, it follows from \cite[Theorem 4.1]{SeitzClassical} that the restriction of $L_G(\lambda)$ to $H^\circ$ is irreducible, and that its highest weight is the sum of all fundamental dominant weights of $H^\circ$. Hence $L_G(\lambda) \downarrow K[u^2] = N' \cdot J_{2^{k-1}}$ by Lemma \ref{lemma:steinbergmodule_jordanblocks}, where $N' = \dim L_G(\lambda) / 2^{k-1}$. Thus $u$ cannot act on $L_G(\lambda)$ with any Jordan blocks of size $< 2^k$. Since $u$ has order $2^k$, we conclude that $L_G(\lambda) \downarrow K[u] = N \cdot J_{2^k}$, where $N = \dim L_G(\lambda) / 2^k$.\end{proof}

For the rest of this section, let $K$ be an algebraically closed field of characteristic $p = 2$. Set $V = K^{10}$ and define $\beta: V \times V \rightarrow K$ to be the non-degenerate alternating bilinear form given by the anti-diagonal matrix $$\begin{pmatrix}
0  &          & 1 \\
  & \reflectbox{$\ddots$} & \\
1 &          & 0 \\
\end{pmatrix}.$$ We consider $$G = \Sp_{10}(K) = \{ A \in \GL_{10}(K) : \beta(Av, Aw) = \beta(v,w) \text{ for all } v, w \in V\}.$$ Note that $G$ is a simple linear algebraic group and simply connected of type $C_5$. Our counterexample is given by two non-conjugate unipotent elements of $G$, so we will have to first make some remarks about unipotent conjugacy classes in $G$.

The unipotent conjugacy classes of the symplectic group in characteristic $2$ were described by Hesselink in \cite{Hesselink}. Let $u \in G$ be a unipotent element of $G$ and set $X = u-1$. The results of Hesselink \cite[Theorem 3.8, Section 3.9]{Hesselink} show that the conjugacy class of $u$ in $G$ is uniquely determined by the tuple $$({d_1}_{\chi(d_1)}^{n_1}, \ldots, {d_t}_{\chi(d_t)}^{n_t}),$$ where $1 \leq d_1 < \cdots < d_t$ are the Jordan block sizes of $u$ on $V$, with block size $d_i$ having multiplicity $n_i$, and $$\chi(m) = \min \{ n \geq 0 : \beta(X^n v, X^{n+1} v) = 0 \text { for all } v \in \operatorname{Ker} X^{m} \}.$$

Our counterexample is given by the pair of unipotent classes of $G$ corresponding to the tuples $(2_{1}^2, 6_3)$ and $(2_{0}^2, 6_3)$, and we will demonstrate that unipotent elements in these two conjugacy classes act similarly on every rational irreducible $KG$-module $L_G(\lambda)$. 

Before proving this, we give representatives for these two unipotent conjugacy classes in $\Sp_{10}(2)$. Denote by $E_{i,j}$ the $10 \times 10$ matrix with $1$ at the $(i,j)$ position and $0$ elsewhere.  The generators of $\Sp_{10}(2) < G$ from \cite{GEN_classical1} are given by $\Sp_{10}(2) = \langle A, B \rangle$, where $A = I + E_{1,5} + E_{1,10} + E_{6,10}$ and $B = E_{2,1} + E_{3,2} + E_{4,3} + E_{5,4} + E_{10,5} + E_{9,10} + E_{8,9} + E_{7,8} + E_{6,7} + E_{1,6}$.

With a computer search, we have found that the elements $$u = B^4 A B^6 A B^5 A$$ and $$u' = B A B^2 A B^4 A B^3 A$$ lie in the unipotent conjugacy classes of $G$ associated with the tuples $(2_1^2, 6_3)$ and $(2_0^2, 6_3)$, respectively. We now give our main result.

\begin{lause}\label{thm:mainresult}
Let $G$ and $u, u' \in G$ be as above. Then $u$ and $u'$ are not conjugate in $G$, but $f(u)$ and $f(u')$ are conjugate in $\GL(V)$ for every rational irreducible representation $f: G \rightarrow \GL(V)$.
\end{lause}

\begin{proof}By Steinberg's tensor product theorem \cite[Theorem 6.1]{SteinbergREP}, it is clear that it will be enough to prove that $u$ and $u'$ act similarly on $L_G(\lambda)$ for all $2$-restricted $\lambda \in X(T)^+$. Furthermore, it follows from \cite[Theorem 11.1]{SteinbergREP} that when $G$ is simple of type $C_l$, we have $$L_G(a_1\varpi_1 + \cdots + a_{l-1}\varpi_{l-1} + a_l \varpi_l) \cong L_G(a_1\varpi_1 + \cdots + a_{l-1}\varpi_{l-1}) \otimes L_G(a_l \varpi_l)$$ for all $a_i \in \{0,1\}$. Consequently, it will be enough to check that $u$ and $u'$ act similarly on $L_G(\varpi_5)$ and $L_G(a_1 \varpi_1 + a_2 \varpi_2 + a_3 \varpi_3 + a_{4}\varpi_{4})$ for all $a_i \in \{0,1\}$. In most cases, we do this by a computer calculation (see Section \ref{section:lietypecomp}), and we have given the Jordan block sizes of $u$ and $u'$ acting on these irreducible $KG$-modules in Table \ref{table:action_counterexample}. There are also alternative approaches in some cases, which we explain in Remark \ref{remark:otherways} below.

First note that the fact that $u$ and $u'$ act on $L_G(\varpi_1 + \varpi_2 + \varpi_3 + \varpi_4)$ with all Jordan blocks of size $8$ follows from Lemma \ref{lemma:char2typeDsteinberg}. For the rest of the cases, we have computed the Jordan block sizes of $u$ and $u'$ on $L_G(\lambda)$ by computing their Jordan block sizes on the corresponding irreducible $\mathbb{F}_2[\Sp_{10}(2)]$-module, which we also denote here by $L_G(\lambda)$ for clarity. 

On the natural module $V \cong L_G(\varpi_1)$ both $u$ and $u'$ act with Jordan blocks $(2^2, 6)$. For $2 \leq i \leq 5$, the exterior power $\wedge^i(V)$ has highest weight $\varpi_i$ and the weight $\varpi_i$ occurs with multiplicity $1$. With a computation in MAGMA, we find $L_G(\varpi_i)$ as the unique composition factor of its dimension in $\wedge^i(V)$. Here we use the list of dimensions of irreducible $KG$-modules with $2$-restricted highest weight provided by Frank L\"ubeck in \cite{LubeckWebsite}, these dimensions are also listed in Table \ref{table:action_counterexample} for convenience. By ``finding $L_G(\varpi_i)$'', we mean that we find matrices representing the action of the generators $A$ and $B$ of $\Sp_{10}(K)$ on $L_G(\varpi_i)$, which then allows us to compute the corresponding matrices for $u$ and $u'$ and their Jordan normal forms. In this manner, one can verify that $u$ and $u'$ act on $L_G(\varpi_i)$ with Jordan blocks as given in Table \ref{table:action_counterexample}.

For $1 \leq i < j \leq 4$, we compute $L_G(\varpi_i + \varpi_j)$ as the unique composition factor of its dimension in the tensor product of $L_G(\varpi_i)$ and $L_G(\varpi_j)$, this is again justified by the character of the tensor product and the list of dimensions. Finally, in a similar manner $L_G(\varpi_i + \varpi_j + \varpi_k)$ for $1 \leq i < j < k \leq 4$ is found as the unique composition factor of its dimension in the tensor product of $L_G(\varpi_i)$ and $L_G(\varpi_j + \varpi_k)$. This is computationally more demanding for $L_G(\omega_2 + \omega_3 + \omega_4)$, but this case can be treated differently, see Remark \ref{remark:otherways}.\end{proof}

\begin{remark}\label{remark:moregeneralconjecture}We do not know whether our example gives a counterexample to Conjecture \ref{conjecture:conjecture_more_general}, and it does not seem obvious that this should be the case. Indeed, given a rational $KG$-module $Z$ which is not semisimple, the amount of information given by the action of a unipotent element on the composition factors of $Z$ seems to be rather limited. Thus it is not clear how to move from the irreducible case to the general case.

We also note here that one can find non-irreducible $K[\Sp_{10}(2)]$-modules $Z$ where $u$ and $u'$ above act with different Jordan block sizes. For example, several such $Z$ can be found as a submodule of $V \otimes V \otimes V$, where $V$ is the natural module as above. However, it turns out in this case that $Z$ is not a $K[\Sp_{10}(4)]$-module, so in particular $Z$ is not a $K[\Sp_{10}(K)]$-module.\end{remark}

\begin{remark}\label{remark:otherways}There are also different ways to deduce that $u$ and $u'$ act similarly on $L_G(\lambda)$ for some of the $\lambda \in X(T)^+$ considered in the proof of Theorem \ref{thm:mainresult}. For example, since $u$ and $u'$ both act with Jordan blocks $(2^2, 6)$ on $V$, it follows that $u$ and $u'$ act similarly on the exterior power $\wedge^i(V)$ for all $1 \leq i \leq 5$. From the fact that $\wedge^i(V)$ is a tilting $G$-module of highest weight $\varpi_i$ \cite[Appendix A]{DonkinTiltingSP}, we conclude that for all $1 \leq i \leq 5$, the elements $u$ and $u'$ act similarly on the indecomposable tilting $G$-module with highest weight $\varpi_i$. Now using the fact that a tensor product of two tilting modules is tilting \cite{Mathieu}, one finds that $u$ and $u'$ act similarly on every tilting $G$-module.

In particular, we conclude that $u$ and $u'$ act similarly on $L_G(\lambda)$ if $L_G(\lambda)$ is tilting. It is well known that $L_G(\lambda)$ is tilting if and only if $\dim L_G(\lambda) = \dim V_G(\lambda)$, and this equality can be checked with the dimensions of $L_G(\lambda)$ provided in \cite{LubeckWebsite} and the dimension formula for the Weyl module $V_G(\lambda)$, which has been implemented in MAGMA. We have marked the cases where $L_G(\lambda)$ is tilting in Table \ref{table:action_counterexample} with a star (*).\end{remark}

With further computations, using the method from Section \ref{subsec:chevalleycomp}, we can verify that in the case of unipotent elements, the counterexample to Conjecture \ref{conjecture:steinberg} given by Theorem \ref{thm:mainresult} is essentially the only one for simple $G$ of rank at most $5$. More precisely, we have the following result.

\begin{lause}\label{thm:rankatmost5}
Let $G$ be simple of rank at most $5$ and let $u, u' \in G$ be unipotent elements. Suppose that $f(u)$ and $f(u')$ are conjugate in $\GL(V)$ for every rational irreducible representation $f: G \rightarrow \GL(V)$. Then either $u$ and $u'$ are conjugate in $G$, or the following hold:

\begin{enumerate}[\normalfont (i)]
\item $p = 2$ and there exists an isogeny $\varphi: \Sp_{10}(K) \rightarrow G$, and
\item the unipotent conjugacy classes in $\varphi^{-1}(u)$ and $\varphi^{-1}(u')$ are those associated with the tuples $(2_1^2, 6_3)$ and $(2_0^2, 6_3)$.
\end{enumerate}

\end{lause}

\begin{proof}Without loss of generality we can assume that $G$ is simple of adjoint type. Our proof consists mostly of a computation with MAGMA (see Section \ref{subsec:chevalleycomp}) for all possible types of $G$. For the exceptional types that we need to consider (type $G_2$ and type $F_4$) the claim already follows from Theorem \ref{thm:exceptional_computations}, so what remains is to consider the classical types.

In the case where $G$ is simple of type $A_l$ with $1 \leq l \leq 5$, by a computation one can verify that except for one pair in the case where $l = 3$ and $p = 2$, all pairs of unipotent conjugacy classes are separated by $L_G(\varpi_1 + \varpi_l)$ or $L_G(\varpi_2 + \varpi_{l-1})$. The pair of unipotent conjugacy classes in the case where $l = 3$ and $p = 2$ is separated by $L_G(4 \varpi_1)$.

Suppose that $G$ is simple of type $B_l$ with $3 \leq l \leq 5$. When $p > 2$, the unipotent conjugacy classes are separated by the natural representation $L_G(\varpi_1)$, see \cite[Proposition 2 of Chapter II]{Gerstenhaber}. If $p = 2$, by \cite[Theorem 28]{SteinbergNotesAMS} there exists an exceptional isogeny $\varphi: G' \rightarrow G$ where $G'$ is simple of type $C_l$, so the result will follow from the type $C_l$ case which we treat next.

Let $G$ be simple of type $C_l$ with $2 \leq l \leq 5$. When $p > 2$, with a computation one can verify that all pairs of unipotent conjugacy classes are separated by $L_G(\varpi_2)$ or $L_G(2\varpi_1)$. For $p = 2$, one computes that except for the pair of classes corresponding to (ii) in the theorem, all pairs of conjugacy classes are separated by some $L_G(\varpi_i)$ ($i$ even) or $L_G(2\varpi_i)$ ($i$ odd).

Finally for $G$ simple of type $D_l$ with $4 \leq l \leq 5$, every pair of unipotent conjugacy classes is separated by $L_G(2\varpi_1)$, $L_G(\varpi_2)$, or $L_G(2\varpi_3)$; except for one pair which occurs in the case where $p = 2$ and $l = 5$. This pair of unipotent conjugacy classes is separated by $L_G(4\varpi_4)$.\end{proof}

\FloatBarrier

\section{Acknowledgements}

The main results of this paper were obtained during my doctoral studies at École Polytechnique Fédérale de Lausanne, supported by a grant from the Swiss National Science Foundation (grant number $200021 \_ 146223$). I am very grateful to Donna Testerman, James E. Humphreys, Gunter Malle, and two anonymous referees for helpful comments and suggestions on earlier versions of this paper.



\end{document}